\documentclass[12pt,reqno]{amsart}
\usepackage{amssymb}

\newtheorem{theorem}{Theorem}[section]

\newtheorem{lemma}[theorem]{Lemma}
\newtheorem{question}[theorem]{Question}

\numberwithin{equation}{section}

\begin{document}

\title[Connections and restrictions to curves]{Connections and restrictions to curves}

\author[I. Biswas]{Indranil Biswas}

\address{School of Mathematics, Tata Institute of Fundamental
Research, Homi Bhabha Road, Mumbai 400005, India}

\email{indranil@math.tifr.res.in}

\author[S. Gurjar]{Sudarshan Gurjar}

\address{Department of Mathematics, Indian Institute of Technology, Mumbai 400076, 
India}

\email{sgurjar@math.iitb.ac.in}



\begin{abstract}
We construct a vector bundle $E$ on a smooth complex projective surface $X$
with the property that the restriction of $E$ to any smooth closed curve in $X$
admits an algebraic connection while $E$ does not admit any algebraic connection.
\end{abstract}

\maketitle

\section{Introduction}

Let $X$ be an irreducible smooth complex projective variety with cotangent bundle
$\Omega^1_X$ and $E$ a vector bundle on $X$. The coherent sheaf of local sections of $E$ will
also be denoted by $E$. 
A connection on $E$ is a $k$-linear homomorphism of sheaves $D\,:\, E\,\longrightarrow
\,E \otimes \Omega^1_X$ satisfying the Leibniz identity which says that $D(fs)\,=\,fD(s)+ s
\otimes df$, where $s$ is a local section of $E$ and $f$ is a locally defined regular function.

Consider the sheaf of differential operators ${\rm Diff}^i_X(E, E)$, of order $i$ on $E$, and the
associated symbol homomorphism $\sigma\, :\, {\rm Diff}^1_X(E, E)\,\longrightarrow\, \text{End}(E)
\otimes TX$. The inverse image
$$
\text{At}(E)\,:=\, \sigma^{-1}(\text{Id}_E\otimes TX)
$$
is the Atiyah bundle for $E$. The resulting short exact sequence
\begin{equation}\label{e1}
0\,\longrightarrow\, {\rm Diff}^0_X(E, E)\,=\, \text{End}(E) \,\longrightarrow\,
\text{At}(E)\,\stackrel{\sigma}{\longrightarrow}\, TX \,\longrightarrow\,0
\end{equation}
is called the Atiyah exact sequence for $E$. A connection on $E$ is a splitting of \eqref{e1}.
We refer the reader to \cite{At} for the details, in particular, see \cite[p.~187, Theorem 1]{At}
and \cite[p.~194, Proposition 9]{At}.

When $X$ is a complex curve, Weil and Atiyah proved the following \cite{We}, \cite{At}:

A vector bundle $V$ on an irreducible smooth projective curve defined over $\mathbb C$ admits
a connection if and only if the degree of each indecomposable component of $V$ is zero.

This was first proved in \cite{We}; see also \cite[p.~69, T\'HEOR\`EME DE WEIL]{Gr}
for an exposition of it.
The above criterion also follows from \cite[p.~188, Theorem 2]{At}, \cite[p.~201, Theorem 8]{At}
and \cite[Theorem 10]{At}. 

A semistable vector bundle $V$ on a smooth complex projective variety $X$ admits a 
connection if all the rational Chern classes of $E$ vanish \cite[p.~40, 
Corollary~3.10]{Si}. On the other hand, a vector bundle $W$ on $X$ is semistable if 
and only if the restriction of $W$ to a general complete intersection curve,
which is an intersection of hyperplanes of sufficiently large degrees, is semistable
\cite[p.~637, Theorem~1.2]{Fl}, \cite[p.~221, Theorem~6.1]{MR}. On the other hand, any
vector bundle $E$ whose restriction to every curve is semistable actually satisfies
very strong conditions \cite{BB}; for example, if $X$ is simply connected, then $E$
must be of the form $L^{\oplus r}$ for some line bundle $L$.

The following is a natural question to ask:

\begin{question}\label{q1}
Let $E$ be a vector bundle on $X$ such that for every smooth closed curve $C\, \subset\, X$, the
restriction $E\vert_C$ admits a connection. Does $E$ admit a connection?
\end{question}

Our aim is to show that in general the above vector bundle $E$ does not admit a connection.

To produce an example of such a vector bundle, we construct a smooth complex
projective surface $X$ with
${\rm Pic}(X)\,=\, \mathbb{Z}$ such that $X$ admits an ample line bundle $L_0$ with
$H^1(X,\, L_0)\,\not=\, 0$. Since ${\rm Pic}(X)\,=\, \mathbb{Z}$, the ample line bundles
on $X$ are naturally parametrized by positive integers. Let $L$ be the smallest ample line
bundle (with respect to this parametrization) with the property that
$H^1(X,\, L)\,\not=\, 0$. Let $E$ be a nontrivial extension
$$
0 \,\longrightarrow\, L \,\longrightarrow\, E \,\longrightarrow\, {\mathcal O}_X
 \,\longrightarrow\, 0\, .
$$
We prove that the vector bundle ${\rm End}(E)$ has the property that
the restriction of it to every smooth closed curve in $X$ admits a connection, while
${\rm End}(E)$ does not admit a connection; see Theorem \ref{thm1}.

A surface $X$ of the above type is constructed by taking a
hyper-K\"ahler $4$--fold $X'$ with ${\rm Pic}(X') \,=\, \mathbb{Z}$.
Let $Y\, \subset\, X'$ be a smooth ample
hypersurface such that $H^j(X',\, {\mathcal O}_{X'}(Y))\,=\, 0$ for $j\,=\,1,\,2$, and
let $Z$ be a very general ample hypersurface of $X'$ such that
$H^j(X',\,{\mathcal O}_{X'}(Z))\,=\,0$ for $j\,=\,1,\, 2$ and
$H^2(X',\, {\mathcal O}_{X'}(Z-Y))\,=\, 0$. Now take the surface $X$ to be
the intersection $Y\cap Z$.

\section{Construction of a surface}\label{se2}

We will construct a smooth complex projective surface $S$ with Picard 
group $\mathbb{Z}$ that has an ample line bundle $L$ with $H^1(S,\,L) \,\neq\, 0$. 

Let $X$ be a hyper-K\"ahler $4$--fold with Picard group 
$\mathbb{Z}$. For example a sufficiently general deformation of ${\rm Hilb}^2(M)$, where $M$ 
is a polarized $K3$ surface, will have this property. Let $Y\, \subset\, X$ be a smooth ample
hypersurface. Note that the vanishing theorem of Kodaira says
that
\begin{equation}\label{kvy}
H^j(X,\, {\mathcal O}_X(Y))\,=\, 0
\end{equation}
for all $j\,>\,0$, because $K_X$ is trivial \cite{Ko}. 
Let $Z$ be a very general ample hypersurface of $X$ such that
both the line bundles ${\mathcal O}_X(Z)$ and ${\mathcal O}_X(Z-Y)$ are ample.
In view of the vanishing theorem of Kodaira, the ampleness of ${\mathcal O}_X(Z)$ implies
that
\begin{equation}\label{kvz}
H^j(X,\, {\mathcal O}_X(Z))\,=\, 0
\end{equation}
for all $j\,>\,0$, while that of ${\mathcal O}_X(Z-Y)$ implies that
\begin{equation}\label{kvzy}
H^j(X,\, {\mathcal O}_X(Z-Y))\,=\, 0
\end{equation}
for all $j\,>\,0$.
Let $$\iota\,:\, S \,:=\, Y\cap Z \, \hookrightarrow\, X$$ be 
the intersection and $$L\,:=\, {\mathcal O}_X(Y)\vert_S$$
the restriction of it. Note that $L$ is ample.

Let ${\mathcal I}\,:=\, {\mathcal O}_X(-S)\, \subset\, {\mathcal O}_X$ be the ideal
sheaf for $S$. Tensoring the
exact sequence
$$
0 \,\longrightarrow\, {\mathcal I} \,\longrightarrow\, {\mathcal O}_X
\,\longrightarrow\, \iota_* {\mathcal O}_S \,\longrightarrow \, 0
$$
by ${\mathcal O}_X(Y)$ we get an exact sequence
\begin{equation}\label{e2}
 0 \,\longrightarrow\, \mathcal{I}(Y) \,\longrightarrow\, {\mathcal O}_X(Y)
\,\longrightarrow\, 
\iota_*L \,\longrightarrow \, 0\, .
\end{equation}
The natural inclusion of ${\mathcal O}_X(-Z)$ in ${\mathcal O}_X$ and
${\mathcal O}_X(Y-Z)$ together produce an inclusion of ${\mathcal O}_X(-Z)$ in
${\mathcal O}_X\oplus {\mathcal O}_X(Y-Z)$. Consequently, we have an exact sequence 
\begin{equation}\label{e3} 
0\,\longrightarrow\, {\mathcal O}_X(-Z) \,\longrightarrow\, {\mathcal O}_X\oplus 
{\mathcal O}_X 
(Y-Z) \,\longrightarrow\, \mathcal{I}(Y) \,\longrightarrow\, 0\, .
\end{equation}

In view of \eqref{kvy}, the connecting homomorphism 
\begin{equation}\label{e4}
H^1(S,\, L) \,{\longrightarrow}\, H^2(X,\, \mathcal{I}(Y))
\end{equation}
in the long exact sequence of cohomologies associated to \eqref{e2} is an 
isomorphism.

Since the canonical line bundle of $X$ is trivial, Serre duality gives
$$
H^{2+j}(X,\, {\mathcal O}_X(-Z))^*\, =\, H^{2-j}(X,\,{\mathcal O}_X(Z))\, .
$$
So using \eqref{kvz} we conclude that
the left-hand side vanishes for $j\,=\, 0,\, 1$. Again by Serre duality 
$$H^2(X,\,{\mathcal O}_X(Y- Z))^*\,=\, H^2(X,\, {\mathcal O}_X(Z-Y))\,=\, 0$$
(see \eqref{kvzy}). 

Thus in the long exact sequence of cohomologies associated to \eqref{e3}, we have
$$
H^{2}(X,\, {\mathcal O}_X(-Z))\,=\, 0\,=\, H^{2+j}(X,\, {\mathcal O}_X(-Z))\, ,
\ \ \text{ and }\ \ H^2(X,\,{\mathcal O}_X(Y- Z))\,=\, 0\, .
$$
Hence this long exact sequence of cohomologies associated to \eqref{e3} gives an isomorphism
$$
H^2(X,\, {\mathcal O}_X) \, \stackrel{\sim}{\longrightarrow}\, H^2(X,\,\mathcal{I}(Y))\, ;
$$
so combining this with the isomorphism in \eqref{e4} it now follows that
$H^1(S,\, L)$ is isomorphic to $H^2(X,\, {\mathcal O}_X)$. We have $\dim H^2(X,\, {\mathcal O}_X)
\,=\,1$, so
\begin{equation}\label{f1}
\dim H^1(S,\, L)\,=\,1\, .
\end{equation}

By Grothendieck--Lefschetz hyperplane theorem for
Picard group, the restriction map ${\rm Pic}(X) \, \longrightarrow\, 
{\rm Pic}(Y)$ is an isomorphism \cite[Expose\'e XII]{SGA}; in fact, a weaker version given in
\cite[Chapter~IV, p.~179, Corollary~3.2]{Ha} suffices for our purpose.
By the generalized Noether--Lefschetz theorem (see 
\cite[p.~121, Theorem~5.1]{Jo}), the restriction map ${\rm Pic}(Y) \,\longrightarrow\, {\rm 
Pic}(S)$ is also an isomorphism. Thus ${\rm Pic}(S)$ is isomorphic to $\mathbb{Z}$. Combining 
this with \eqref{f1} it follows that the surface $S$ has the desired properties.

\section{Question \ref{q1} in special cases}

In this section will will
first use the construction in Section \ref{se2} to show that Question \ref{q1}
in the introduction has a negative answer in general. Then we will show that in some
particular cases the answer is affirmative.

\subsection{Example with a negative answer}

We will construct a smooth projective surface $X$ and a vector bundle $E$ on it that does not 
admit any connection while the restriction of $E$ to every smooth curve in $X$ admits a 
connection.

Let $X$ be a smooth complex projective surface with ${\rm Pic}(X)\,=\, \mathbb{Z}$ that admits 
an ample line bundle $L$ with $H^1(X,\, L)\,\not=\, 0$; we saw in Section \ref{se2} that such a 
surface exists. Let ${\mathcal O}_X(1)$ denote the ample generator of ${\rm Pic}(X)$. Then $L 
\, =\, {\mathcal O}_X(r)\,=\, {\mathcal O}_X(1)^{\otimes r}$ with $r$ positive. We choose $L$ 
with smallest possible $r$. Since ${\rm Pic}(X)\,=\, \mathbb{Z}$, we have $H^1(X,\, {\mathcal 
O}_X)\,=\, 0$ because $H^1(X,\, {\mathcal O}_X)\,=\, 0$ is the (abelian) Lie algebra
of the Lie group ${\rm Pic}(X)$. On the other hand,
the Kodaira vanishing theorem says that $H^1(X,\, {\mathcal O}_X(-k))\,=\, 
0$ for all $k\, >\, 0$. Therefore, it follows that
\begin{equation}\label{f2}
H^1(X,\, L\otimes {\mathcal O}_X(-d)) \,=\, 0\, , \forall ~ \ d\,>\,0\, . 
\end{equation}
Let
\begin{equation}\label{e5}
 0 \,\longrightarrow\, L \,\longrightarrow\, E \,\longrightarrow\, {\mathcal O}_X
 \,\longrightarrow\, 0
\end{equation}
be the non-split extension corresponding to a non-zero element in $H^1(X,\, L)$.

\begin{theorem}\label{thm1}
The vector bundle ${\rm End}(E)\,=\, E\otimes E^*$ in \eqref{e5} has the property that
the restriction of it to every smooth closed curve in $X$ admits a connection. The
vector bundle ${\rm End}(E)$ does not admit a connection.
\end{theorem}

\begin{proof}
Take any smooth closed curve $C\, \subset\, X$. So
$C \,\in\, |{\mathcal O}_X(d)|$ with $d$ positive.
Consider the restriction homomorphism $H^1(X,\,L) \,\longrightarrow\,
H^1(C,\,L\vert_C)$. Using the long exact sequence of cohomologies associated to
$$
0\,\longrightarrow\, L\otimes {\mathcal O}_X(-d)\,\longrightarrow\,
L \,\longrightarrow\, L\vert_C \,\longrightarrow\, 0
$$
we conclude that its
kernel is $H^1(X,\, L\otimes {\mathcal O}_X(-d))$, which 
is zero by \eqref{f2}. In particular, the extension class for \eqref{e5} has a nonzero image
in $H^1(C,\,L\vert_C)$. Therefore, 
the restriction of the exact sequence \eqref{e5} to $C$ does not split.

We will show that $E\vert_C$ is indecomposable. 

Assume that $E\vert_C\,=\, L_1\oplus L_2$ with $\text{degree}(L_1)\, \geq\,
\text{degree}(L_2)$. Since $\text{degree}(E\vert_C)\,=\, \text{degree}(L\vert_C)
\, >\, 0\,=\, \text{degree}({\mathcal O}_C)$, the composition
$$
L_1\, \hookrightarrow\, E\vert_C \, \longrightarrow\, {\mathcal O}_C
$$
is the zero homomorphism. Hence $L_1$ coincides with the subbundle
$L\vert_C\, \subset\, E\vert_C$. This contradicts the earlier observation that the 
restriction of the exact sequence \eqref{e5} to $C$ does not split. Hence we
conclude that $E\vert_C$ is indecomposable.

Consider the projective bundle ${\mathbb P}(E\vert_C)\, \longrightarrow\, C$. Let 
$E_{\text{PGL}(2)}\, \longrightarrow\, C$ be the principal $\text{PGL}(2, {\mathbb 
C})$--bundle corresponding to it. Since $E$ is indecomposable, it follows that 
$E_{\text{PGL}(2)}$ admits an algebraic connection \cite[p.~342, Theorem~4.1]{AB}.
The vector bundle $\text{End}(E\vert_C)\, \longrightarrow\, C$ is associated
to $E_{\text{PGL}(2)}$ for the adjoint action of $\text{PGL}(2, {\mathbb
C})$ on $\text{End}_{\mathbb C}({\mathbb C}^2)\,=\, \text{M}(2, {\mathbb C})$.
Therefore, a connection on $E_{\text{PGL}(2)}$ induces a connection on
the vector bundle $\text{End}(E\vert_C)$. Hence, we conclude that
$\text{End}(E\vert_C)\,=\, \text{End}(E)\vert_C$ admits an algebraic connection.

On the other hand, $c_2(\text{End}(E))\,=\, - c_1(L)^2\, \not=\, 0$. This implies
that the vector bundle $E$ on $X$ does not admit a connection \cite[Theorem~4]{At}.
\end{proof}

\subsection{Special cases with positive answer}

Let $S$ be a smooth complex projective curve, $X$ a smooth complex projective variety and
$p\, :\, X \,\longrightarrow S$ a smooth surjective morphism such that every fiber of
$p$ is rationally connected. Assume that there is a smooth closed curve
$\widetilde{S}\, \subset\, X$ such that the restriction
$$
p\vert_{\widetilde{S}}\, :\, \widetilde{S}\, \longrightarrow\, S
$$
is an \'etale morphism.

\begin{lemma}\label{lem1}
Let $E$ be a vector bundle on $X$ whose restriction to every smooth curve on
$X$ admits a connection. Then $E$ admits a connection. 
\end{lemma}

\begin{proof}
Let $Y$ be a smooth complex projective rationally connected variety and $V$ a vector 
bundle on $Y$, such that for every smooth rational curve ${\mathbb C}{\mathbb P}^1 
\, \stackrel{\iota}{\hookrightarrow}\, Y$ the restriction $\iota^*V$ has a 
connection. Any connection on a curve is flat, and ${\mathbb C}{\mathbb P}^1$ is 
simply connected, so the above vector bundle $\iota^*V$ is trivial. This implies 
that the vector bundle $V$ is trivial \cite[Proposition~1.2]{BdS}.

{}From the above observation it follows that $E\,=\, p^*p_*E$. Therefore, it
suffices to show that $p_*E$ admits a connection. Now, by the given condition,
the vector bundle $(p\vert_{\widetilde{S}})^*p_*E\, =\, E\vert_{\widetilde{S}}$ admits
a connection. Fix a connection $D$ on $E\vert_{\widetilde{S}}$. Averaging $D$ over the
fibers of $p$ we get a connection on $p_*E$. This completes the proof.
\end{proof}

\section*{Acknowledgements}

We are very grateful to Jason Starr for his generous help. We thank the
referee heartily for going through the paper carefully and providing comments
to improve the exposition.
The first-named author is supported by a J. C. Bose Fellowship.



\begin{thebibliography}{ZZZZ}

\bibitem[AB]{AB} H. Azad and I. Biswas, On holomorphic principal bundles 
over a compact Riemann surface admitting a flat connection, {\it Math. Ann.} {\bf
322} (2002), 333--346.

\bibitem[At]{At} M. F. Atiyah, Complex analytic connections in fibre
bundles, \textit{Trans. Amer. Math. Soc.} \textbf{85} (1957), 181--207.

\bibitem[BB]{BB} I. Biswas and U. Bruzzo, On semistable principal bundles over a complex 
projective manifold, {\it Int. Math. Res. Not.} (2008), article ID rnn035.

\bibitem[BdS]{BdS} I. Biswas and J. P. P. dos Santos, On the vector bundles 
over rationally connected varieties, Com. Ren. Math. Acad. Sci. Paris {\bf 347}
(2009), 1173--1176.

\bibitem[Gr]{Gr} A. Grothendieck, Sur le m\'emoire de Weil. G\'en\'eralisation des
fonctions ab\'eliennes, {\it S\'eminaire Bourbaki}, Volume 4 (1956-1958) , Talk no. 141,
p. 57--71.

\bibitem[SGA2]{SGA} A. Grothendieck and M. Raynaud Cohomologie locale des faisceaux cohr\'ents
et th\'eor\`emes de Lefschetz locaux et globaux, (SGA 2), Documents Mathématiques (Paris), 4,
Paris: Soci\'et\'e Math\'ematique de France, 1968, arXiv:math/0511279.

\bibitem[Fl]{Fl} H. Flenner, Restrictions of semistable bundles on projective 
varieties, {\it Comment. Math. Helv.} {\bf 59} (1984), 635--650.

\bibitem[Ha]{Ha} R. Hartshorne, {\it Ample subvarieties of algebraic varieties},
Lecture notes in mathematics, No. 156, Springer-Verlag, Berlin, 1970.

\bibitem[Jo]{Jo} K. Joshi, A Noether-Lefschetz theorem and 
applications, {\it Jour. Algebraic Geom.} {\bf 4} (1995), 105--135.

\bibitem[Ko]{Ko} K. Kodaira, On a differential-geometric method in the theory 
of analytic stacks, {\it Proc. Natl. Acad. Sci. USA} {\bf 39} (1953), 1268--1273.

\bibitem[MR]{MR} V. B. Mehta and A. Ramanathan, Semistable
sheaves on projective varieties and their restriction to
curves, {\it Math. Ann.} \textbf{258} (1982), 213--224.

\bibitem[Si]{Si} C. Simpson, Higgs bundles and local systems, \textit{Inst. Hautes
\'Etudes Sci. Publ. Math.} \textbf{75} (1992), 5--95.

\bibitem[We]{We} A. Weil, G\'en\'eralisation des fonctions ab\'eliennes,
\textit{Jour. Math. Pure Appl.} \textbf{17} (1938), 47--87.

\end{thebibliography}
\end{document}